\documentclass[12pt]{amsart}  
\usepackage[utf8]{inputenc}
\usepackage[T1]{fontenc}
\usepackage{textcomp}
\usepackage{lmodern}
\usepackage{setspace}
\usepackage{fourier}
\usepackage[margin=1in]{geometry}

\usepackage[english]{babel}
 
\newtheorem{theorem}{Theorem}
\newtheorem{corollary}{Corollary}[theorem]
\newtheorem{lemma}[theorem]{Lemma}
\newtheorem{proposition}{Proposition}
\usepackage{fourier}
\usepackage{mathtools}

\usepackage{hyperref}       
\usepackage{url}            
\usepackage{booktabs}       
\usepackage{amsfonts}       
\usepackage{nicefrac}       
\usepackage{lipsum}

\usepackage[usenames,dvipsnames]{color}

\title{The Distribution of Integral Points on the Wonderful Compactification by Height}

\author{Dylon Chow}

\begin{document}
\maketitle

\begin{abstract}
We study the asymptotic distribution of $S$-integral points of bounded height on partial bi-equivariant compactifications of semi-simple groups of adjoint type.
\end{abstract}

\section{Introduction.}

\subsection{Rational Points by Height}

Let $X \subset \mathbb{P}^n$ be a smooth projective variety over a number field $F$. Fix a height function

$$H:\mathbb{P}^n(F) \rightarrow \mathbb{R}_{>0}$$

and consider the counting function

$$N(X,B)=\# \{x \in X(F):H(x) \leq B\}.$$

For certain varieties $X$, Manin's conjecture and its refinements predict precise asymptotic formulas for $N(X^\circ,B)$ as $B \rightarrow \infty$, where $X^\circ \subset X$ is an appropriate Zariski open subset of $X$. For an overview of progress in the case of surfaces, see ... When the variety is a bi-equivariant compactification, approaches from dynamics and harmonic analysis on adele groups can be used. For example, the papers \cite{JAMS} and \cite{GMO} prove Manin's conjecture for wonderful compactifications of semisimple groups.

\subsection{Integral points.}

Let $X$ be a smooth projective variety over a number field $F$, and let $D \subset X$ be a divisor with strict normal crossings. Choose models $\mathcal{X}, \mathcal{D}$ of $X,D$ over the ring of integers $\mathcal{O}_F$ of $F$. A rational point $x \in X(F)$ gives rise to a section $\pi_x$:Spec$(\mathcal{O}_F) \rightarrow \mathcal{X}$ of the structure morphism. Choose a finite set of places $S$ of $\mathcal{O}_F$ that contains the archimedean places. A $(\mathcal{D},S)$-integral point is a section $\pi_x$ such that for all $v \notin S$ one has

$$\pi_{x,v} \cap \mathcal{D}_v = \emptyset,$$

i.e. $\pi_x$ avoids $\mathcal{D}$ over Spec$(\mathcal{O}_F) \backslash S$. 

Let $K_X$ denote the canonical divisor on $X$. Assume that $-(K_X+D)$ is ample and that $(\mathcal{D},S)$-integral points are Zariski dense. For a Zariski open $X^\circ \subset X$, let

$$N(X^\circ,-(K_X+D),B)=\#\{x|H_{-(K_X+D)} \leq B\}$$

be the counting function on the set of $(\mathcal{D},S)$-integral points in $X^\circ(F)$ with respect to $-(K_X+D)$. An asymptotic formula for the number of integral points by height has been established for toric varieties (\cite{inttoric}) and compactifications of vector groups (see \cite{integralvector}). The proofs use harmonic analysis on adele groups. In these cases, the underlying groups are abelian, and the harmonic analysis reduces to the classical theory of Fourier analysis on locally compact abelian groups. In \cite{TTB} an asymptotic formula was established when $X$ is the wonderful compactification of a split adjoint group.

In this paper we remove the split condition and establish analytic properties of height zeta functions for integral points on partial equivariant compactifications of semisimple adjoint groups. To state our result, let $G$ be a semisimple adjoint group of rank at least 2 and let $X$ be the wonderful compactification of $G$. Assume that $S$ and $D$ are such that the set of $(D,S)$-integral points is Zariski dense. Let $T$ be a maximal split torus of $G$ and let $E$ be a number field over which $T$ splits.

Chambert-Loir and Tschinkel (see \cite{CLT}) defined, for each place $v \in S$, a simplicial complex $\mathcal{C}^{an}_{F_v}(D)$, which is called the analytic Clemens complex. It encodes the incidence properties of the $v$-adic manifolds given by the irreducible components of $D$. Its dimension is equal to the maximal number of irreducible components of $D$ defined over the local field $F_v$ whose intersection has $F_v$-points, minus one.




    
    

For

$$\lambda=\sum_{\alpha \in \Delta(G_E,T_E)}\lambda_\alpha \alpha$$

in Eff$^0(X)$, the interior of the effective cone of $X$, put

$$a(\lambda)=\max(\max_{\alpha \in \mathcal{A}_D}\dfrac{\kappa_\alpha}{\lambda_\alpha},\max_{\alpha \notin \mathcal{A}_D}\dfrac{\kappa_\alpha+1}{\lambda_\alpha}).$$ Here $\kappa_\alpha$ is defined as in section 2.3.




Let $d_v(\lambda)$ be the dimension of the simplicial complex $\mathcal{C}^{\text{an}}_{F_v,\lambda}(X \backslash D)$ obtained from $\mathcal{C}^{\text{an}}_{F_v}(X \backslash D)$ by removing all faces containing a vertex such that $\kappa_\alpha < a(\lambda)\lambda_\alpha$.

Let

$$b(\lambda)=\text{rk} \ \text{Pic}(X \backslash D)+\sum_{v \in S} d_v(\lambda)$$

and

\begin{center}
    $$c=\lim_{s \rightarrow a(\lambda)}\dfrac{(s-a(\lambda))^{b(\lambda)}}{a(\lambda)(b(\lambda)-1)!} \cdot |\chi_{S,D,\lambda}(G)| \cdot \int_{G(\mathbb{A})^{\text{Ker}_\lambda}}\delta_{S,D}(g)H(s\lambda,g)^{-1} \ dg>0.$$
\end{center}

Here $|\chi_{S,D,\lambda}(G)|$ is the finite number of automorphic characters that contribute to the right-most pole.

\begin{theorem}
 With the notation above, the number of $(S,D)$-integral points of bounded height on $X$ with respect to $\lambda$ is asymptotic to 

$$c B^{a(\lambda)}\log(B)^{b(\lambda)-1}(1+o(1)), \ \ \ \ \ \ B \rightarrow \infty,$$
\end{theorem}






\subsection{Outline of paper and method of proof.}

The proof relies on a strategy developed in the earlier papers on rational points and integral points for split groups. We introduce the height zeta function

\[Z(s)=\sum_{x \in G(F) \cap \mathcal{U}(\mathcal{O}_{F,S})}H(x)^{-s}\]

for a complex parameter $s$. By a Tauberian theorem, it is sufficient to establish certain analytic properties of the function $Z(s)$. In particular, we prove its convergence for Re($s)>a(\lambda)$ and then establish its meromorphic continuation in some half-plane Re($s)>a(\lambda)-\delta$. This is achieved by rewriting the expression in terms of the spectral decomposition of $L^2(G(F) \backslash G(\mathbb{A}))$. One of the terms in the spectral expansion is a sum, over a finite set of automorphic characters, of products of local $v$-adic integrals and adelic integrals. After establishing the analytic properties of the adelic integrals and $v$-adic integrals for almost all places, the author learned that integrals of more general type had been analyzed by Chambert-Loir and Tschinkel in \cite{CLT}. In fact, these general results have been applied to establish asymptotic formulas for integral points on vector groups and toric varieties. In contrast to those cases, for compactifications of semisimple groups only finitely many automorphic characters appear in the sum. Nevertheless, the proof given in this paper, being an adaptation of the original proof in \cite{JAMS}, may be of interest.

For the semisimple case, unlike the case of commutative groups, automorphic forms like Eisenstein series have to be considered. As in the case of rational points, uniform estimates need to be established. In \cite{JAMS}, a technical assumption on local height functions was made to simplify the computations, but a paper of Loughran, Tanimoto, and Takloo-Bighash (\cite{loughran}) showed how to remove this restriction. We use their results in this paper.

The rest of the paper is organized as follows. In section 2, we review some background and set up notation. In section 3 we write out a spectral expansion for the height zeta function and establish the required analytic properties.


\section{Background and Notation.}

\subsection{Haar measures and absolute values.}  In this paragraph, let $F$ be a local field of characteristic zero, i.e., either $\mathbb{R}, \mathbb{C}$, or a finite extension of the field $\mathbb{Q}_p$ of $p$-adic numbers. Fix a Haar measure $\mu$ on $F$. Its modulus is an absolute value $|.|$ on $F$, defined by $\mu(a\Omega)=|a|\mu(\Omega)$ for any $a \in F$ and any measurable subset $\Omega \subset F$. 

\subsection{Local and global fields.}  Let $F$ be a number field. Let Val($F)$ be the set of normalized absolute values of $F$. For $v \in $Val($F)$, we let $F_v$ be the $v$-adic completion of $F$. Its absolute value is denoted by $|.|_v$. Let $v$ be a finite place of $F$. We denote by $\mathcal{O}_v$ the ring of $v$-adic integers, by $\mathfrak{m}_v$ its unique maximal ideal. The residue field $\mathcal{O}_v/\mathfrak{m}_v$ is denoted by $k_v$, and we write $q_v$ for its cardinality. We fix a uniformizing element $\varpi$; one has $|\varpi|_v=q_v^{-1}$. We fix an algebraic closure $\overline{F}$ of $F$. and denote the absolute Galois group of $F$ by $\Gamma_F$. We denote by $\mathbb{A}$ the ring of adeles $\mathbb{A}_F$ of the field $F$. It is a locally compact topological ring and carries a Haar measure $\mu$. The quotient space $\mathbb{A}_F/F$ is compact; we normalize $\mu$ so that $\mu(\mathbb{A}_F/F)=1$.

\subsection{Algebraic groups.}

Let $G$ be a semisimple adjoint $F$-group. Let $T$ be a maximal torus in $G$. Let $E$ be a finite Galois extension over $F$ such that $T$ splits over $E$, and let $\Gamma$ denote Gal$(E/F)$. For each place $v$ of $F$, let $S_v$ be the maximal split subtorus of $T_{F_v}$ in $G_{F_v}$. Let $\mathfrak{X}^*(S_v)$ denote the set of characters of $S_v$ defined over $F_v$, and let $\Phi(G_{F_v},S_v)$ denote the set of roots of $S_v$ in $G_{F_v}$. A choice of minimal parabolic in $G_{F_v}$ determines a set $\mathfrak{X}^+(S_v)$ of positive characters, a set $\Phi(G_{F_v},S_v)$ of positive roots in $\mathfrak{X}^*(S_v)$, and a set $\Delta(G_{F_v},S_v)$ of simple roots. Choose a Borel subgroup in $G_E$ and denote by $\Delta(G_E,T_E)$ the corresponding set of simple roots. We define constants $\kappa_\alpha$ for $\alpha \in \Delta(G_E,T_E)$ by 

\[ \sum_{\alpha>0,\alpha \in \Phi(G_E,T_E)}\alpha=\sum_{\alpha \in \Delta(G_E,T_E)}\kappa_\alpha \alpha.\]

For a simple root $\theta \in \Delta(G_{F_v},S_v),$ define a cocharacter $\overset{\vee}\theta$ by \[\langle{\alpha,\overset{\vee}\theta \rangle}=-\delta_{\alpha \theta},\] 

where $\delta_{\alpha \theta}$ is 1 if and only if $\alpha=\theta$. 







\subsubsection{Archimedean places.} Let $H$ be a Lie group and $\mathfrak{h}$ its Lie algebra. Let $U(\mathfrak{h})$ denote the universal enveloping algebra. We regard it as a ring of differential operators on $C^\infty(H)$. Let $G_\infty=\prod_{v \ \text{arch}} G(F_v)$ and $K_\infty = \prod_{v \ \text{arch}} K_v$. Let $C_G$ and $C_K$ be the Casimir operators of the Lie algebras of $G_\infty$ and $K_\infty$, respectively. Let $\Delta$ denote the operator $C_G-2C_K$. It is the admissible Laplacian with respect to the Killing form, in the sense of Borel-Garland (\cite{borelgarland}). The operator $\Delta$ is elliptic, and so all of its eigenvalues are non-negative real numbers.










\subsubsection{Cartan decomposition}

If $v$ is archimedean we set $F_v^0=\{x \in \mathbb{R}:x \geq 0\}$ and $\hat{F}_v=\{x \in \mathbb{R}:x \geq 1\}$. If $v$ is non-archimedean, we fix a uniformizer $\varpi_v$ of $F_v$, and set

$$\hat{F}_v=\{\varpi^{-n}:n \in \mathbb{N}\}, \ \ \  F_v^0=\{\varpi^n:n \in \mathbb{Z}\},$$

$$S_v(F_v)^+=\{a \in S_v(F_v):\alpha(a) \in \hat{F}_v \ \mathrm{for} \ \mathrm{each} \ \alpha \in \Phi^+(S_v)\}$$ 

\begin{theorem} (Bruhat-Tits) For each place $v$, there is a maximal compact subgroup $K_v$ of $G(F_v)$ and a finite set $\Omega_v \subset G(F_v)$ such that

$$G(F_v)=K_v S_v(F_v)^+ \Omega_v K_v.$$

That is, for each $g \in G(F_v)$, there exist unique elements $a \in S_v(F_v)^+$ and $d \in \Omega_v$ such that $g \in K_v adK_v$.
\end{theorem}

We remark that if $F_v$ is archimedean, or if $G$ is unramified (i.e., quasisplit, and splits over an unramified extension), then $G(F_v)=K_v S_v(F_v)^+ K_v$.

\subsubsection{The $*$-action} In this subsection, let $k$ be an arbitrary field of characteristic 0. Let $G$ be a connected reductive group over $k$. Let $S$ be a maximal split torus in $G$ and $T$ a maximal torus containing $S$. Then $G_{k^s}$ is split, and so $T_{k^s}$ is contained in a Borel subgroup $B$ of $G_{k^s}$. Let $(X,\Phi,\Phi^\vee,\Delta)$ be the based root datum of $(G_{k^s},B,T_{k^s})$. As the action of $\Gamma=$Gal$(k^s/k)$ preserves $T_{k^s}$, $\Gamma$ acts naturally on $X=X^*(T)$ and $X^\vee=X_*(T)$. These actions preserve $\Phi$ and $\Phi^\vee$. We recall the $*$-action. If $\sigma$ is an element of $\Gamma$, then there is a unique element $w_\sigma$ in the Weyl group  $W(G_{k^s},T_{k^s})$ such that $w_\sigma(\sigma(\Delta))=\Delta$. The $*$-action of $\Gamma$ on $\Delta$ is defined by $\sigma * \alpha=w_\sigma(\sigma \alpha)$, for $\alpha \in \Delta$. We refer to $\Gamma$ acting by the $*$-action as $\Gamma^*$. If $G$ is split, then the $*$-action is trivial. If $G$ is quasi-split, then the $*$-action is the restriction to $\Delta$ of the natural action of $\Gamma$ on $X^*(T)$. 

If $k'$ is a field extension of $k$ over which $T$ splits, then we identify $X^*(T_{k^s})$ with $X^*(T_{k'})$, etc. Restriction from $T_{k'}$ to $S_{k'}$ defines a surjective homomorphism  res$:X^*(T_{k'}) \rightarrow X^*(S_{k'})$. Let $\Phi(G_{k'},T_{k'})$ be the associated root system and $\Phi(G_k,S):=\text{res}(\Phi)-\{0\}$ be the restricted root system (in general not reduced). The set of simple roots $\Delta(G_{k'},T_{k'}) \subseteq \Phi(G_{k'},T_{k'})$ determines a set of simple restricted roots

$$\Delta_k=\mathrm{res}(\Delta) \backslash \{0\}.$$

(Here $0$ denotes the trivial character on $S$.) The restriction map $\Delta \rightarrow \Delta_k \cup \{0\}$ is surjective. Let $\Delta_0$ denote the set of elements of $\Delta$ that restrict to the trivial character on $S$. The fibers of the restriction map $\Delta-\Delta_0 \rightarrow \Delta_k$ are precisely the orbits of the $*$-action on $\Delta-\Delta_0$. Furthermore, $G$ is quasi-split over $k$ if and only if $\Delta_0$ is trivial (\cite{milne}, 25.28), in which case the number of $\Gamma^*$-orbits is precisely the cardinality of $\Delta_k$.

We will apply these considerations when $k=F_v$ and $k'=E_w$, where $F$ and $E$ are number fields, and $v$ and $w$ are places of $F$ and $W$, respectively, such that $w|v$.

\subsection{The Wonderful Compactification}

If $G$ is an algebraic group over a field $k$ and $H \subseteq G$ is a subgroup scheme, then an equivariant compactification of the homogeneous space $G/H$ is defined to be a proper $G$-scheme equipped with an open equivariant immersion $G/H \rightarrow X$ with schematically dense image. If char$(k)=0$, then by the existence of equivariant desingularizations, every homogeneous space has a smooth projective equivariant compactification. \par

Suppose that $G$ is a semisimple adjoint group over an algebraically closed field. In this case, De Concini and Procesi (\cite{deconcini}) constructed an equivariant compactification $X$ of $G$, equipped with an action of the group $G \times G$ that extends the action of $G \times G$ on $G$ by left and right multiplication. In other words, $X$ is a bi-equivariant compactification of the homogeneous space $G=(G \times G)/\mathrm{diag}(G)$, where diag$(G)$ denotes the diagonal in $G \times G$. De Concini and Springer (\cite{deconcinispringer}) extended the construction over arbitrary fields. \par

We summarize some of the properties that will be used, working initially over an algebraically closed field. Let $B$ be a Borel subgroup of $G$ and let $T$ be a maximal split torus of $G$ contained in $B$. Let $r$ denote the rank of $G$, and denote the simple roots by $\alpha_1,...,\alpha_r$. The compactification $X$ is smooth, and the boundary $X \backslash G$ is the union of $r$ nonsingular prime divisors with normal crossings. The weight lattice of $G$ is isomorphic to the Picard group Pic$(X)$. The isomorphism defined in \cite{bk}(Prop 6.1.11) identifies the boundary divisors with the simple roots, and so Pic$(X)$ is a free abelian group of rank $r$, generated by the fundamental weights. We shall write $D_{\alpha_1},...,D_{\alpha_r}$ for the corresponding boundary divisors. \par

Now we consider the case over number fields. Let $G$ be semisimple adjoint over a number field $F$, and let $X$ be the corresponding wonderful compactification. The Galois group $\Gamma=\mathrm{Gal}(\overline{F}/F)$ acts on Pic$(X_{\overline{F}})$. The bijection between the set of simple roots and the boundary divisors if $\Gamma$-equivariant, and so Pic$(X)$ is freely generated by the line bundles corresponding to the orbits of the simple roots under the $*$-action. The $F$-irreducible boundary components of $X$ are the divisors of the form 

$$D_J=\sum_{\alpha \in J}D_\alpha$$

for $\Gamma^*$-stable subsets $J \subset \Delta(G_{\overline{F}},T_{\overline{F}})$.

We denote the set of boundary divisors of $X$ by $\mathcal{A}$. The closed cone Eff$(X) \subset$ Pic$(X)_{\mathbb{R}}$ of effective divisors on $X$ is generated by the boundary components of $X$, i.e.,

\begin{center}
Eff$(X)=\bigoplus_{\alpha \in \mathcal{A}}\mathbb{R}_{\geq 0}D_\alpha$.
\end{center}

A dominant weight $\lambda$ is called regular if $\lambda=\sum_{\alpha \in \Delta}m_\alpha \omega_\alpha$ with all $m_\alpha>0$ where $\{\omega_\alpha:\alpha \in \Delta\}$ is the set of fundamental weights. The globally generated line bundles correspond to the dominant weights, and the ample line bundles correspond to the regular dominant weights. 
An anticanonical divisor for $X$ is given by

$$-K_X = \sum_\alpha (\kappa_\alpha+1)D_\alpha.$$


\subsection{Height functions}


Endow the line bundles $\mathcal{O}_X(D_\alpha)$, for $\alpha \in \mathcal{A}$, with smooth adelic metrics $||.||_\alpha$. Let $f_\alpha$ be the canonical section of the line bundle $\mathcal{O}_X(D_\alpha)$ whose divisor is equal to $D_\alpha$. 

For each place $v$ of $F$, define a map $H_v$ as follows: for $\textbf{s} \in \mathbb{C}^\mathcal{A} \cong \mathrm{Pic}(X)_\mathbb{C}$ and $g \in G(F_v)$, we set

$$H_v(\textbf{s},g_v)=\prod_{\alpha \in \mathcal{A}}||f_\alpha(g_v)||_v^{-s_\alpha}.$$

Now define the ``height pairing" $H:\mathrm{Pic}(X)_\mathbb{C} \times G(\mathbb{A}) \rightarrow \mathbb{C}$ by

$$H(\textbf{s},(g_v)_v)=\prod_v H_v(\textbf{s},g_v).$$

When $\textbf{s}$ correspond to a very ample class $\lambda$ in Pic$(X)$, the restriction of $H(\textbf{s},.)$ to $G(F)$ is the standard (exponential) height relative to the projective embedding of $X$ defined by $\lambda$. For each real number $B$, the set of $x$ in $G(F)$ such that $H(\textbf{s},x) \leq B$ is finite.

\subsubsection{Reducing to the simple case.} The adjoint group $G$ decomposes into simple factors, $G=G_1 \times ... \times G_m$ (\cite{milne}). Each $G_i$ is also adjoint, and the wonderful compactification of $G$ is $X=X_1 \times ... \times X_m$, where $X_i$ denotes the wonderful compactification of $G_i$. By the Bruhat decomposition, $G_i$ is geometrically rational (see, for instance, \cite{linearization}, Prop. 5.1.3), and hence so is $X_i$. Therefore (\cite{linearization}, Prop. 5.1.2) Pic$((X_1)_{\overline{F}} \times...\times (X_m)_{\overline{F}})=\bigoplus_{i=1}^m \mathrm{Pic}((X_i)_{\overline{F}})$. The height functions on the product are expressed as products of height functions (\cite{hindry}, exercise F. 15). Therefore, we will assume that $G$ is simple.





\subsection{Eisenstein series}


The height zeta function $Z(s)$ will be expressed in terms of Eisenstein series. We review the general definition of Eisenstein series, closely following \cite{arthur}.

\vspace{3mm}

1. Let $G$ be a reductive group over $F$. If $v$ is finite, define $K_v$ to be $G(\mathcal{O}_v)$ if this latter group is a special maximal compact subgroup of $G(F_v)$. This takes care of almost all $v$. For the remaining finite $v$, we let $K_v$ be any fixed special maximal compact subgroup of $G(F_v)$. We also fix a minimal parabolic subgroup $P_0$, defined over $F$, and a Levi component $M_0$ of $P_0$. Let $A_0$ be the maximal split torus in the center of $M_0$. For each archimedean place $v$ of $F$, we choose a maximal compact subgroup $K_v \subset G(F_v)$ such that

$$G(F_v)=K_v A_0(F_v)K_v.$$

We set $K=\prod K_v$. It is a maximal compact subgroup of $G(\mathbb{A})$. We also set $K_0=\prod_{v<\infty} K_v$ and $K_\infty=\prod_{v|\infty}K_v$.

\vspace{3mm}
2. Fix a parabolic subgroup $P$ defined over $F$ that contains $P_0$. Let $N=N_P$ be the unipotent radical of $P$. Let $M_P$ be the unique Levi component of $P$ that contains $M_0$. Then the split component, $A_P$, of the center of $M_P$ is contained in $A_0$. Let $X^*(M_P)_F$ be the group of characters of $M_P$ defined over $F$, and define

\begin{center}
$\mathfrak{a}_{M_P}=$Hom$(X^*(M_P)_F,\mathbb{R})$.
\end{center}

We have a map

$$H_{M_P}:M_P(\mathbb{A}) \rightarrow \mathfrak{a}_{M_P}$$ 

defined by

$$\mathrm{exp}\langle{H_M(m),\chi\rangle}=\prod_v |\chi(m_v)|_v$$

for $\lambda \in X^*(M_P)_F$.

Let $M_P(\mathbb{A})^1$ be the kernel of the homomorphism $H_{M_P}$. Using the Iwasawa decomposition $G(\mathbb{A})=N_P(\mathbb{A}) M_P(\mathbb{A}) K$, we define a morphism

$$H_P:G(\mathbb{A}) \rightarrow \mathfrak{a}_{M_P}$$

by

$$nmk \mapsto H_{M_P}(m)$$

for $(n,m,k) \in N(\mathbb{A}) \times M_P(\mathbb{A}) \times K$.

\vspace{3mm}
3. Let $W$ be the restricted Weyl group of $(G,A_0)$. Then $W$ acts on the dual space of $\mathfrak{a}_0$. For a pair of standard parabolic subgroups $P, P_1$, let $W(\mathfrak{a}_P,\mathfrak{a}_{P_1})$ be the set of distinct linear isomorphisms from $\mathfrak{a}_P$ onto $\mathfrak{a}_{P_1}$ obtained by restricting elements in $W$ to $\mathfrak{a}_P$. Two parabolic subgroups $P$ and $P_1$ are said to be associated if $W(\mathfrak{a}_P,\mathfrak{a}_{P_1})$ is not empty. This defines an equivalence relation on the set of parabolic subgroups in $G$. If $P$ and $P_1$ are associate, then $M=M_1$. Moreover, $w$ preserves $M=M_1$, where $w \in W(\mathfrak{a}_P,\mathfrak{a}_{P_1})$. In view of this fact, it is therefore natural to expect a relationship between representations of $G$ induced from $P$ and those induced from $P'.$ 

Suppose that $P$ is a parabolic subgroup. There is a finite number of disjoint open subsets of $\mathfrak{a}_P$, called the chambers of $\mathfrak{a}_P$. We shall write $n(A_P)$ for the number of chambers.














\vspace{3mm}
4. Let $M$ be the Levi factor of some standard parabolic $P$ of $G$. Let $$L^2_{\mathrm{cusp}}(M(F) \backslash M(\mathbb{A}_F)^1)$$ be the space of functions $\phi$ in $L^2(M(F) \backslash M(\mathbb{A}_F)^1)$ such that for every parabolic $Q \subseteq P$ we have

$$\int_{N_Q(F) \cap M(F) \backslash N_Q(\mathbb{A}) \cap M(\mathbb{A})}\phi(nm) \ dn=0$$

for almost all $m$. There is a $G(\mathbb{A})$-invariant orthogonal decomposition

$$L^2_{\mathrm{cusp}}(M(F) \backslash M(\mathbb{A})^1)=\bigoplus_\sigma L^2_{\mathrm{cusp},\sigma}(M(F) \backslash M(\mathbb{A})^1),$$

where $\sigma$ ranges over irreducible unitary representations of $M(\mathbb{A})^1$, and $$L^2_{\mathrm{cusp},\sigma}(M(F) \backslash M(\mathbb{A}))$$ is $M(\mathbb{A})$-isomorphic to a finite number of copies of $\sigma$. An irreducible unitary representation of $M(\mathbb{A})^1$ is said to be cuspidal if $L^2_{\mathrm{cusp},\sigma}(M(F) \backslash M(\mathbb{A}))\neq 0$.

\vspace{3mm}
5. We define an equivalence relation on the set of pairs $(M,\rho)$ with $M$ a Levi factor of some standard parabolic subgroup of $G$ and $\rho$ is an irreducible unitary representation of $M(\mathbb{A})^1$ occurring in $L^2_{\mathrm{cusp}}(M(F) \backslash M(\mathbb{A})^1)$. A \textit{cuspidal automorphic datum} is defined to be an equivalence class of pairs $(P,\sigma)$, where $P \subset G$ is a standard parabolic subgroup of $G$, and $\sigma$ is an irreducible representation of $M_P(\mathbb{A})^1$ such that the space 

$$L^2_{\mathrm{cusp},\sigma}(M_P(F) \backslash M_P(\mathbb{A})^1)$$ is nonzero. 

The restricted Weyl group $W$ of $(G,A_0)$ acts naturally on $\mathfrak{a}_{P_0}$ and $\mathfrak{a}_{P_0}^*$. For any $s \in W$, fix a representative $w_s$ in the intersection of $G(F)$ with the normalizer of $A_0$. The equivalence relation is defined as follows: $(P',\sigma')$ is equivalent to $(P,\sigma)$ if there is an element $s \in W(\mathfrak{a}_P, \mathfrak{a}_{P'})$ such that the representation

$$s^{-1}\sigma':m \mapsto \sigma'(w_s m w_s^{-1}),  \ \ \ \ \ \ \ \ \ \ m \in M_P(\mathbb{A})^1,$$

of $M_P(\mathbb{A})^1$ is unitarily equivalent to $\sigma$. We write $\mathfrak{X}$ of the set of cuspidal automorphic data $\chi=\{(P,\sigma)\}$.

For any $\chi \in \mathfrak{X}$ let $\mathcal{P}_\chi$ denote the class of associated parabolic subgroups consisting of those parabolic subgroups $P$ with a Levi subgroup $M$ and a representation $\rho$ such that $(M,\rho) \in \chi$.




\vspace{3mm}
If $M$ is the Levi factor of some parabolic subgroup and $\chi \in \mathfrak{X}$, set

$$L^2_{\text{cusp}}(M(F) \backslash M(\mathbb{A})^1)_\chi=\bigoplus_{(\rho:(M,\rho) \in \chi)}V_\rho.$$

This is a closed subspace of $L^2_{\text{cusp}}(M(F) \backslash M(\mathbb{A})^1)$. It is zero if $P \notin \mathcal{P}_\chi$ for every parabolic subgroup $P$ that has $M$ as a Levi factor. Then we have

$$L^2_{\text{cusp}}(M(F) \backslash M(\mathbb{A})^1)=\bigoplus_{\chi \in \mathfrak{X}}L^2_{\mathrm{cusp}}(M(F) \backslash M(\mathbb{A})^1)_\chi.$$

\vspace{3mm}
6. Any class $\chi=\{(\mathcal{P},\sigma)\}$ in $\mathfrak{X}$ determines an associated class of standard parabolic subgroups. For any $P$, let $\Pi(M_P)$ denote the set of equivalence classes of irreducible unitary representations of $M_P(\mathbb{A})$. If $\zeta \in \mathfrak{a}_\mathbb{C}^*$ and $\pi \in \Pi(M_P)$, let $\pi_\zeta$ be the product of $\pi$ with the quasi-character


$$x \mapsto e^{\zeta(H_P(x))}, \ \ \ x \in G(\mathbb{A}).$$

If $\zeta$ belongs to $i\mathfrak{a}_P^*$, then $\pi_\zeta$ is unitary, and so we obtain a free action of the group $i\mathfrak{a}_P^*$ on $\Pi(M_P)$. Then $\Pi(M_P)$ becomes a differentiable manifold whose connected components are the orbits of $i\mathfrak{a}_P^*$. We can transfer our Haar measure on $i\mathfrak{a}_P^*$ to each of the orbits in $\Pi(M_P)$. This allows one to define a measure $d\pi$ on $\Pi(M_P)$.

\vspace{3mm}
7. Let $\pi \in \Pi(M_P)$. We will define certain representations of $G(\mathbb{A})$ induced from $\pi$. First we describe the space of the representation. Let $H_P^0(\pi)$ be the space of smooth functions
\[\phi:N_P(\mathbb{A})M_P(F) \backslash G(\mathbb{A}) \rightarrow \mathbb{C}\] which satisfy the following conditions:

(i) $\phi$ is right $K$-finite, i.e., the span of the set of functions $\phi_k:x \mapsto \phi(xk), x \in G(\mathbb{A})$, indexed by $k \in K$, is finite dimensional.

(ii) For every $x \in G(\mathbb{A})$, the function $m \mapsto \phi(mx), m \in M_P(\mathbb{A})$
is a matrix coefficient of $\pi$.

(iii) If we define $||\phi||^2$ by the equation $$||\phi||^2 = \int_K \int_{M_P(F) \backslash M_P(\mathbb{A})^1}|\phi(mk)|^2 \ dm \ dk,$$ then $||\phi||^2<\infty$.

Let $H_P(\pi)$ be the Hilbert space completion of $H_P^0(\pi)$. Now we describe the action of $G(\mathbb{A})$ on this space. More precisely, we will define representations $I_P(\pi_\zeta)$ of $G(\mathbb{A})$, one for each $\zeta \in \mathfrak{a}_{P,\mathbb{C}}^*$, as follows. For each $y \in G(\mathbb{A})$, $I_P(\pi_\zeta)(y)$ maps a function $\phi$ of $H_P(\pi)$ to the function given by

$$(I_P(\pi_\zeta)(y)\phi)(x)=\phi(xy)e^{(\zeta+\rho_P)(H_P(xy))}e^{-(\lambda+\rho_P)(H_P(x))}.$$

For $\phi \in H_P(\pi)$ and $\zeta \in \mathfrak{a}_{\mathbb{C}}^*$, define

$$\phi_\zeta(x)=\phi(x)e^{\zeta(H_P(x))}, \ \ x \in G(\mathbb{A}).$$

Recall that the function $e^{\rho_P(H_P(.))}$ is the square root of the modular function of the group $P(\mathbb{A})$. It is included in the definition so that the representation $I_P(\pi_\zeta)$ is unitary whenever the inducing representation is unitary, which is to say, whenever $\zeta$ belongs to the subset $i\mathfrak{a}_P^*$ of $\mathfrak{a}_{P,\mathbb{C}}^*$. The above equation can be rewritten as

$$(I_P(\pi_\zeta)(y))(\phi)(x)=\phi_\zeta(xy)\delta_P(xy)^{1/2}\delta_P(x)^{-1/2}.$$

We have put the twist by $\zeta$ into the operator $I_P(\pi_\zeta)(y)$ rather than the underlying Hilbert space $H_P(\pi)$, so that $H_P(\pi)$ is independent of $\zeta$. We remark that $H_P(\pi)=\{0\}$ unless there is a subrepresentation of the regular representation of $M(\mathbb{A})^1$ on $L^2(M(F) \backslash M(\mathbb{A})^1)$ which is equivalent to the restriction of $\pi$ to $M(\mathbb{A})^1$.

\vspace{3mm}
8. Note that $I_P(\pi_\zeta)$ is a representation of $G(\mathbb{A})$ on a space of functions on $G(\mathbb{A})$, but the functions in the representation space $H_P(\pi)$ are not left invariant under $G(F)$. The functions in the space are left invariant under $P(F)$, however. We will average the functions to make them invariant under $G(F)$. We are now ready to define Eisenstein series. They provide an intertwining map from $I_P(\pi_\zeta)$ to the regular action of $G(\mathbb{A})$ on functions on $G(\mathbb{A})$.

Suppose that $\pi \in \Pi(M)$. For $\phi \in H_P^0(\pi)$ and $\zeta \in (\mathfrak{a}_P^*)_{\mathbb{C}}$, we formally define

$$E(x,\phi, \zeta)=\sum_{\gamma \in P(F) \backslash G(F)} \phi_\zeta(\gamma x) \delta_P(\gamma x)^{1/2}.$$




This expression is defined by a sum over a noncompact space. In general, such an expression does not converge. For any $P$, we can form the chamber

$$(\mathfrak{a}_P^*)^+=\{\Lambda \in \mathfrak{a}_P^*:\Lambda(\alpha^\vee)>0 \ \text{for} \ \text{all} \ \alpha \in \Delta_P\}$$

in $\mathfrak{a}_P^*.$ Here $\Delta_P$ is the set of roots of $(P,A_P)$, where $A_P$ is the split component of the center of $M_P$. As before, suppose that $\pi \in \Pi(M), \phi \in H_P^0(\pi)$, and $\zeta \in \mathfrak{a}^*_{P,\mathbb{C}}$. It is a theorem of Langlands that if $\zeta$ lies in the open subset of $\mathfrak{a}_{P,\mathbb{C}}^*$ with Re$(\zeta) \in \rho_P+(\mathfrak{a}_P^*)^+$, then the sum that defines $E(x,\phi,\zeta)$ converges absolutely to an analytic function of $\zeta$.

The set of points $\zeta$ for which $I_P(\zeta)$ is unitary, i.e. such that $\zeta$ belongs to the real subspace $i\mathfrak{a}_P^*$ of $\mathfrak{a}_{P,\mathbb{C}}^*$, is never inside the domain of absolute convergence of the Eisenstein series. It is a theorem of Langlands that the series $E(x,\phi,\zeta)$ have analytic continuations to this space. More precisely, if $\phi \in H_P^0(\pi)$, then $E(x,\phi,\zeta)$ can be analytically continued to a meromorphic function of $\zeta \in \mathfrak{a}_{P,\mathbb{C}}^*$. If $\zeta \in i \mathfrak{a}_P^*$, then $E(x,\phi,\zeta)$ is analytic. We denote the value of this analytically continued function at $\zeta=0$ by $E(x,\phi)$.

\vspace{3mm}
9. We need more notation to write down a spectral decomposition of the height zeta function. Suppose that $P$ is fixed and that $\chi \in \mathfrak{X}$. Suppose first of all that there is a group $P_1$ in $\mathcal{P}$ which is contained in $P$. Let $\psi$ be a smooth function on $N_{P_1}(\mathbb{A})M_{P_1}(F)\backslash G(\mathbb{A})$ such that

$$\Psi_a(m,k):=\psi(amk), k \in K, m \in M_{P_1}(F)\backslash M_{P_1}(\mathbb{A})^1, a \in A_{P_1}(F)\backslash A_{P_1}(\mathbb{A})$$

vanishes for $a$ outside a compact subset of $A_{P_1}(F)\backslash A_{P_1}(\mathbb{A})$, transforms under $K_\infty$ according to an irreducible representation $W$, and, as a function of $m$, belongs to $L^2_{\text{cusp}}(M_{P_1}(F)\backslash M_{P_1}(\mathbb{A})^1)$. The function

$$\hat{\psi}^M(m):=\sum_{\delta \in P_1(F)\cap M_P(F) \backslash M_P(F)}\psi(\delta m), \ \ \ m \in M_P(F)\backslash M_P(\mathbb{A})^1,$$

is square integrable on $M_P(F)\backslash M_P(\mathbb{A})^1$. We define 

$$L^2(M_P(F)\backslash M_P(\mathbb{A})^1)_\chi$$

to be the closed span of all functions of the form $\hat{\psi}^M$, where $P_1$ runs through those groups in $\mathcal{P}$ which are contained in $P$, and $W$ is allowed to vary over all irreducible representations of $K_\infty$. If there does not exist a group $P_1 \in \mathcal{P}$ which is contained in $P$, define $L^2(M_P(F)\backslash M_P(\mathbb{A})^1)_\chi$ to be $\{0\}$. Then we have an orthogonal direct sum decomposition

$$L^2(M_P(F)\backslash M_P(\mathbb{A})^1)=\bigoplus_{\chi \in \mathfrak{X}} L^2(M_P(F)\backslash M_P(\mathbb{A})^1)_\chi.$$

Given $\chi \in \mathfrak{X}$, let $H_P(\pi)_\chi$ be the closed subspace of $H_P(\pi)$ consisting of those $\phi$ such that for all $x$ the function

$$m \mapsto \phi(mx), m \in M_P(F) \backslash M_P(\mathbb{A})^1$$

belongs to $L^2(M_P(F) \backslash M_P(\mathbb{A})^1)_\chi$. Then

$$H_P(\pi)=\bigoplus_{\chi \in \mathfrak{X}} H_P(\pi)_\chi.$$

Suppose that $W$ is an equivalence class of irreducible representations of $K_\infty$. Let $H_P(\pi)_{\chi,K_0}$ be the subspace of functions in $H_P(\pi)_\chi$ which are invariant under $K_0 \cap K$, and let $H_P(\pi)_{\chi,K_0,W}$ be the space of functions in $H_P(\pi)_{\chi,K_0}$ which transform under $K_\infty$ according to $W$. Each of the spaces $H_P(\pi)_{\chi,K_0,W}$ is finite-dimensional. We shall need orthonormal bases of the spaces $H_P(\pi)_\chi$. We fix such a basis, $B_P(\pi)_\chi$, for each $\pi$ and $\chi$, in such a way that for every $\zeta \in i\mathfrak{a}^*$,

$$B_P(\pi_\zeta)_\chi=\{\phi_\zeta:\phi \in B_P(\pi)_\chi \},$$

and that every $\phi \in B_P(\pi)_\chi$ belongs to one of the spaces $H_P(\pi)_{\chi,K_0,W}$.

\subsection{Clemens complexes}

For each place $v$ of $F$, we denote the $v$-adic analytic Clemens complex of $D$ by $\mathcal{C}^{an}_{F_v}(D)$. We refer to section 3.1.6 of \cite{integralvector} for the relevant definitions. Its faces correspond to irreducible components of intersections of irreducible components of $D_{F_v}$, which contain $F_v$-rational points.

\subsection{Notation}

\subsubsection{Maximal compact subgroups.} For each nonarchimedean place $v$ of $F$, there exists a compact open subgroup $K_v \subset G(F_v)$ such that for all $L$, $H_{L,v}$ is bi-$K_v$-invariant. Moreover, one may take $K_v=G(\mathcal{O}_v)$ for all but finitely many $v$ (\cite{JAMS}, Prop. 6.3.). Let $K=\prod_v K_v$, and let $K_0 = \prod_{v <\infty}K_v$.

\subsubsection{Measures.} For each $v \in \text{Val}(F)$, let $dg_v$ denote the Haar measure on $G(K_v)$ such that $K_v$ has volume 1 whenever $v < \infty$. Then the collection $\{dg_v:v \in \text{Val}(F)\}$ defines a Haar measure, say $\mu$, on $G(\mathbb{A})$. Since $G$ is semisimple, then $G(F)$ is a lattice in $G(\mathbb{A})$. Thus, by replacing $dg_v, v \in F_\infty$, with a suitable multiple of it, we may assume that $\mu(G(F) \backslash G(\mathbb{A}))=1$. 

\subsubsection{Modular quasicharacter.} For a parabolic subgroup $P$ of a connected reductive group over a nonarchimedean local field $F$, the modular quasi-character by $\delta_P:P(F) \rightarrow \mathbb{R}^\times$ is defined by $$\delta_P(p)=|\mathrm{det} \mathrm{Ad}_{\mathfrak{p}}(p):\mathfrak{p}(F) \rightarrow \mathfrak{p}(F))|.$$

\subsubsection{Domains of convergence.} Let $\mathcal{T}_G$ be the subset of $\mathbb{C}^r$ consisting of vectors $(s_1,...,s_r)$ such that $s_i=s_j$ whenever $\alpha_i$ and $\alpha_j$ are in the same Galois orbit. For $\epsilon \in \mathbb{R}$, let $\mathcal{T}_\epsilon$ denote the set of $\textbf{s}=(s_\alpha)_\alpha \in \mathcal{T}_G$ such that Re$(s_\alpha)>\kappa_\alpha+1+\epsilon$, for all $\alpha$. For each subset $R$ of $\mathbb{C}$, we set $\mathcal{T}(R)$ to be the collection of $\textbf{s}=(s_\alpha)_\alpha$ with $s_\alpha \in R$ for all $\alpha$. We set
$\mathcal{T}_\epsilon^D=\{\textbf{s}:\text{Re}(s_\alpha)>\kappa_\alpha+1+\epsilon,\ \text{for all} \ \alpha \notin \mathcal{A}_D\}.$

\subsubsection{Automorphic characters.} Let $\mathcal{X}$ be the set of unitary continuous $G(F)$-invariant homomorphisms $G(\mathbb{A}) \rightarrow \mathbb{C}^\times$ which are invariant under $K$ on both sides. This set is finite (see \cite{GMO}, Lemma 4.7 (2).).

\subsubsection{Integral points.} Let $S$ be a finite set of places of $F$ containing the archimedean places. Fix a model $\mathcal{U}$ over the ring $\mathfrak{o}_{F,S}$ of $S$-integers in $F$. The $S$-integral points of $U$ are the elements of $\mathcal{U}(\mathfrak{o}_{F,S})$, in other words, those rational points of $U(F)$ which extend to a section of the structure morphism from $\mathcal{U}$ to Spec$(\mathfrak{o}_{F,S})$. For any finite place $v$ of $F$ such that $v \notin S$, let $\mathfrak{u}_v=\mathcal{U}(\mathfrak{o}_{F,v})$, and let $\delta_v$ be the characteristic function of the subset $\mathfrak{u}_v \subset X(F_v)$. A point $x \in X(F)$ is an $S$-integral point of $U$ if and only if $x \in \mathfrak{u}_v$ for every place $v$ of $F$ such that $v \notin S$. This condition is equivalent to the condition that $\prod_{v \notin S}\delta_v(x)=1$. For $v \in S$ we set $\delta_v \equiv 1$. For $g \in G(\mathbb{A})$, we define $$\delta_{D,S}(g)=\prod_{v \notin S}\delta_v(g_v).$$

By the definition of an adelic metric, there exists a finite set of places $S_F$ and a flat projective model $\mathcal{X}$ over $\mathfrak{o}_{F,S_F}$ such that:

\begin{itemize}
    \item $\mathcal{X}$ is a smooth equivariant compactification of the $\mathfrak{o}_{F,S_F}$-group scheme $G$;
    
    \item for each $\alpha \in \mathcal{A}$, the closure $\mathcal{D}_\alpha$ of $D_\alpha$ is a divisor on $\mathcal{X}$;
    
    \item the boundary $\mathcal{X} \backslash G$ is the union of these divisors $\mathcal{D}_\alpha$;
    
    \item for each $\alpha$, the section $f_\alpha$ extends to a global section of the line bundle $\mathcal{O}(\mathcal{D}_\alpha)$ on $\mathcal{X}$ whose divisor is precisely $\mathcal{D}_\alpha$.
\end{itemize}

For places $v \notin S_D \cup S_F$, $\mathfrak{u}_v=\mathcal{U}(\mathfrak{o}_v)$.

\section{Heights on the Wonderful Compactification}

\subsection{Spectral expansion.}

To study the distribution properties of $(D,S)$-integral points we will establish analytic properties of the height zeta function

$$Z_{S,D}:\mathcal{T}_G \times G(\mathbb{A}) \mapsto \mathbb{C},$$

which is defined by

$$Z_{D,S}(\textbf{s},g)=\sum_{\gamma \in G(F)}\delta_{D,S}(g)H(\textbf{s},\gamma g)^{-1}.$$

\begin{proposition}
Given $g \in G(\mathbb{A})$, the series defining $Z_{D,S}(.,g)$ converges absolutely to a holomorphic function for $\textbf{s} \in \mathcal{T}_{\gg 0}$. For all $\textbf{s}$ in the region of convergence,

$$Z_{S,D}(\textbf{s},.) \in C^\infty(G(F) \backslash G(\mathbb{A})),$$

and for all integers $n \geq 1$ and all $\partial \in U(\mathfrak{g}), \partial^n Z_{S,D}(\textbf{s}.) \in L^2$. Moreover, in this domain, we have an equality

\begin{equation}Z_{S,D}(\textbf{s},g)=\sum_{\chi \in \mathfrak{X}} \sum_P \dfrac{1}{n(A_P)}\int_{\Pi(M_P)}\int_{G(\mathbb{A})}\delta_{S,D}(g')H(\textbf{s},g')^{-1}(\sum_{\phi \in B_P(\pi)_\chi}E(g,\phi)\overline{E(g',\phi)} ) dg' d\pi.
\end{equation}
\end{proposition}

\begin{proof}
Identical to the proof of \cite{JAMS}, Proposition 8.2; it suffices to observe that $Z_{S,D}$ is a subsum of the series defining the height zeta function for rational points considered in \cite{JAMS}.
\end{proof}

\begin{lemma}
Let $\chi$ be a unitary character $G(\mathbb{A}) \rightarrow \mathbb{C}^*$ that is not bi-K-invariant. Then

$$\int_{G(\mathbb{A})}\delta_{S,D}(g)H(s\lambda,g)^{-1}\chi(g) \ dg=0$$

for all $s$ in the domain of convergence and all $\lambda$.
\end{lemma}

\begin{proof}
See \cite{GMO}, Lemma 4.7(2).
\end{proof}

When we set $g=e$, the identity in $G(\mathbb{A})$, we obtain

\begin{equation}
    Z(\textbf{s},e)=\sum_{\chi \in \mathcal{X}}\prod_v \int_{G(F_v)}\delta_v(g_v)H_v(\textbf{s},g_v)^{-1}\chi_v(g_v) \ dg_v +S^\flat(\textbf{s}),
\end{equation}

where $S^\flat(\textbf{s})$ denotes the subsum in Equation (1) corresponding to infinite dimensional representations (restricted to $g=e$). The innermost sum in the definition of $S^\flat(\textbf{s})$ is uniformly convergent for $g$ in compact sets (see the first half of the proof of Lemma 4.4 of \cite{JAMS}). Therefore, we may interchange the innermost summation with the integral over $G(\mathbb{A})$ and find that $S^\flat(\textbf{s})$ equals

$$ \sum_{\chi \in \mathfrak{X}}^\flat \sum_P \dfrac{1}{n(A_P)}\int_{\Pi(M_P)}(\sum_{\phi \in B_P(\pi)_\chi}E(e,\phi)\int_{G(\mathbb{A})}\delta_{S,D}(g)H(\textbf{s},g)^{-1}\overline{E(g,\phi)} \ dg) \ d\pi_P.$$

\subsection{Complexified Height Function.}

We will need a local characterization of integrality in terms of roots. Recall that $E$ is a splitting field of $T$. Given a boundary divisor $D_\alpha$ of the wonderful compactification of $G_E$, the local heights are given at all places $w$ by

$$H_{D_\alpha,w}(g)=|\alpha(t)|_w,$$

where $g=k_1 t k_2$ for $k_1$ and $k_2$ in $G(\mathcal{O}_{E_w})$. Over $E$, we use the following local characterization of integrality $(S,D_\alpha)$-integrality: at all places outside of $S$, the $w$-adic distance to the boundary component $D_\alpha$ is 1. That is, intuitively, we require the contributions to the ``$w$-adic closeness" of the point $P$ to $D_\alpha$ to be in $S$.

The set-theoretic characteristic function $\delta=\delta_{S,D}$ of the set of $(S,D)$-integral points on $G_E$ is defined as follows: For $w \notin S$, let

\[ \delta_w(g_w) =
  \begin{cases*}
    1 \quad& if $\alpha(t_w)=1$ for all $\alpha \in \mathcal{A}_D$ \\
    0 & if otherwise \\
  \end{cases*}\]

where $g_w=k_w t_w k'_w$ with $t_w \in T(E_w), k_w,k'_w \in K_w$.

For $w \in S$, we put $\delta_v \equiv 1$ and write

$$\delta=\prod_v \delta_v.$$

A point $\gamma \in G(E) \subset G(\mathbb{A}_E)$ is $(D,S)$-integral if $\delta(\gamma)$ is equal to 1.

Now we address our situation, in which $G$ is not split over $F$. We recall the local properties of exponential heights:

$$H_{D,v}(P)=(\prod_{w|v}H_{D,w}(P)^{[E_w:F_v]})^{\frac{1}{[E:F]}}.$$

For almost all places $v$, $g_v \in G(F_v)$ can be written as $k_v t_v k'_v$ with $t_v \in S_v(F_v)^+$ and $k_v,k'_v \in K_v$, and
 $$H_{D_\alpha,v}(g_v)=H_{D_\alpha,v}(t_v)=|\alpha(t_v)|_v$$
 
 by the degree formula. Thus, for such places (which we denote by $S_F$) $\delta_v$ is given by
 
 $$\delta_v(g_v)=\prod_{\alpha \in \mathcal{A}_D}|\alpha(t_v)|_v.$$

 The local complexified height function on $G(F_v)$ has the following explicit form:

$$H_v:\mathcal{T}_{G} \times G(F_v) \rightarrow \mathbb{C}$$

$$(\textbf{s},g_v)\mapsto \prod_{\alpha \in \Delta(G_{E_w},T_{E_w})}|\alpha(t_v)|_v^{s_\alpha}.$$

Given $\theta_v \in \Delta(G_{F_v},S_v)$, let $l_v(\theta_v)$ denote the number of $\beta \in \Delta(G_{E_w},T_{E_w})$ with $r_u(\beta)=\theta_v$. For each $\theta_v \in \Delta(G_{F_v},S_v)$, there is a simple root $\beta$ in $\Delta(G_E,T_E)$ such that $r_v(\iota^*_u(\beta))=\theta_v$. We define local parameters by requiring that $s_{\theta_v}$ depend only on the Galois orbit of $\beta$ in $\Delta(G_{E_w},T_{E_w})$. Then, for $\textbf{s} \in \mathcal{T}_{G}$ and $g_v \in G(F_v)$, we have

\begin{equation}
    H_v(\textbf{s},g_v)=\prod_{\theta_v \in \Delta(G_{F_v},S_v)}|\theta_v(t_v)|_v^{l_v(\theta_v)s_{\theta_v}}.
\end{equation}

For $v \notin S_F$, $\delta_v(g_v)=1$ if and only if, when $g_v$ is written as $k_1 t_v k_2$, we have $\theta_v(t_v)=1$ for all $\theta_v \in \Delta(G_{F_v},S_v)$ such that $\theta_v$ is the restriction of a root of $T_{E_w}$ in a Galois orbit corresponding to $D$.

\subsection{Integrals from one dimensional representations.}

\subsubsection{An L-function}

In this section we review some of the notation defined in section 2.8 of \cite{JAMS}. Fix a simple root $\alpha \in \Delta(G_E,T_E)$ and let $\Gamma_\alpha$ be the stabilizer of $\alpha$ in $\Gamma$. Let $E_\alpha$ be the field of definition of $\alpha$, \textit{i.e.}, the fixed field of $\Gamma_\alpha$. We then have a morphism

$$\overset{\vee}\alpha:\mathbb{G}_m \rightarrow T$$

defined over $E_\alpha$, and consequently a continuous homomorphism

$$\overset{\vee}\alpha_{\mathbb{A}}:\mathbb{G}_m(\mathbb{A}_{E_\alpha}) \rightarrow T(\mathbb{A}_{E_\alpha}).$$

Let $N:T(\mathbb{A}_{E_\alpha}) \rightarrow T(\mathbb{A}_F)$ be the norm map, as defined in \cite{JAMS}, section 1.6. Let $\phi_\alpha$ be the composite

    $$\phi_\alpha=N \circ \overset{\vee}\alpha_{\mathbb{A}}:\mathbb{G}_m(\mathbb{A}_{E_\alpha}) \rightarrow T(\mathbb{A}_{E_\alpha}) \rightarrow T(\mathbb{A}_F).$$

If $\chi$ is a character of $T(\mathbb{A}_F)/T(F)$, then $\xi_\alpha(\chi)=\chi \circ \phi_\alpha$ is a Hecke character. For $w \in \mathrm{Val}(E_\alpha)$, define $\xi_\alpha(\chi)_w$ by $\xi_\alpha(\chi)=\prod_{w \in \text{Val}(E_\alpha)}\xi_\alpha(\chi)_w$.

Let $\alpha \in \Delta(G_E,T_E)$, and let $\mathfrak{O}=\Gamma \cdot \alpha$ be the orbit of $\alpha$. The Hecke L-function $L(s,\xi_\beta(\chi))$ depends only on the Galois orbit $\mathfrak{O}$, and not on the particular $\beta$. For this reason, we denote the $L$-function $L(s,\xi_\beta(\chi))$ by $L(s,\xi_\mathfrak{O}(\chi))$. Suppose $(E_\alpha)_w/F_v$ is unramified and that $\xi_\alpha(\chi)_w$ is unramified. Let

$$L_w(s,\xi_{\alpha}(\chi)_w)=(1-\xi_\alpha(\chi)_w(\varpi_w)q_w^{-s})^{-1},$$

where $\varpi_w$ is a prime element of $(E_\alpha)_w$. Then with the above notations,
$$L_w(s,\xi_\alpha(\chi)_w)=(1-\chi_v(\theta^\vee_v(\varpi_v))q_v^{-l_v(\theta_v)s})^{-1}.$$

\subsubsection{Infinite product}

\begin{theorem}
Let $\chi=\otimes'_v \chi_v$ be a one-dimensional unramified automorphic representation of $G(\mathbb{A})$. 
 There exists a function $f_{S,\chi}$, which depends only on $(s_\alpha)_{\alpha \notin \mathcal{A}_D},$ is holomorphic in $\mathcal{T}^D_{-1/2}$, uniformly bounded in $\mathcal{T}^D_{-1/2+\epsilon}$, for any $\epsilon>0$, and such that

\[\int_{G(\mathbb{A}_{S_D \cup S_F})}\delta_{D,S}(g)H_S(\textbf{s},g)^{-1}\chi(g) \ dg = \prod_{\mathfrak{O} \notin \mathcal{A}_D}L^{S_\mathfrak{O}}(s_\mathfrak{O}-\kappa_\mathfrak{O}, \xi_{\mathfrak{O}}(\chi)) \cdot f_{S,\chi}(\textbf{s}).\]
\end{theorem}

Here, given an orbit $\mathfrak{O}$, there corresponds a subfield of $E$. We let $S_\mathfrak{O}$ be the finite set of places such that if $w \notin S_\mathfrak{O}$, then $w$ does not lie above any place $v \in S_F$.

\begin{proof}
Fix a place $v$ of $F$ such that $v \notin S_D \cup S_F$. For each $\theta_v \in \Delta(G_{F_v},S_v)$, there is a simple root $\beta$ in $\Delta(G_E,T_E)$ such that $r_v(\iota^*_u(\beta))=\theta_v$. We require that $s_{\theta_v}$ depend only on the Galois orbit of $\beta$ in $\Delta(G_E,T_E)$.

Starting with an element $\textbf{s}=(s_\alpha)_\alpha \in \mathcal{T}_{G}$ and $v \notin S_F$, we obtain a tuple $\textbf{s}^v=(s_\theta^v)$ indexed by $\Delta(G_{F_v},S_v)$ by setting $s^v_{r_v(\iota^*(\alpha))}=s_\alpha$; this is well-defined. For $\textbf{s},\textbf{t} \in \mathcal{T}_{G}$ and $v \notin S_F$, we set

$$\langle{\textbf{s},\textbf{t} \rangle}_v=\sum_{\theta_v \in \Delta(G_{F_v},S_v)}s_\theta^v t_\theta^v.$$

For any vector $\textbf{a}=(a_\alpha)_\alpha \in \mathcal{T}(\mathbb{N})$, we set

\[t_v(\textbf{a})=\prod_{\theta \in \Delta(G_v,S_v)}\theta^\vee(\varpi_v)^{a_\theta^v}.\]

By the Cartan decomposition, $$\int_{G(\mathbb{A}_{S_D \cup S_F})}\delta_{D,S}(g)H_S(\textbf{s},g)^{-1}\chi(g) \ dg =
    \sum_{w \in S_v(F_v)^+} \delta_v(w)H_v(\textbf{s},w)^{-1}\chi_v(w)\mathrm{vol}(K_v w K_v).
$$

The above sum is equal to

\[\sum_{\textbf{a} \in \mathcal{T}(\mathbb{N})} \delta_v(t_v(\textbf{a}))q_v^{-\langle{\textbf{s},\textbf{a} \rangle}_v}\chi_v(t_v(\textbf{a}))\mathrm{vol}(K_v t_v(\textbf{a})K_v).\]

This can be rewritten as

\[\sum_{\textbf{a}  \in \mathcal{T}(\mathbb{N})} \delta_v(t_v(\textbf{a}))q_v^{-\langle{\textbf{s}-2\rho,\textbf{a} \rangle}_v}\chi_v(t_v(\textbf{a}))+b_v(\textbf{s}),\] where

\[b_v(\textbf{s})=\sum_{\textbf{a}  \in \mathcal{T}(\mathbb{N})} \delta_v(t_v(\textbf{a}))q_v^{-\langle{\textbf{s},\textbf{a}\rangle}_v}\chi_v(t_v(\textbf{a}))(\mathrm{vol}(K_v t_v(\textbf{a})K_v)-\delta_{B_v}(t_v(\textbf{a}))).\]


\vspace{3mm}










First, we must bound the infinite product

\[\prod_{v \notin S_D \cup S_F} \sum_{\textbf{a}  \in \mathcal{T}(\mathbb{N})} \delta_v(t_v(\textbf{a}))q_v^{-\langle{\textbf{s}-2\rho,\textbf{a} \rangle}_v}\chi_v(t_v(\textbf{a})).\]

This expression is equal to

\[\prod_{v \notin S_D \cup S_F}\prod_{\theta \in \mathcal{A}_v - \mathcal{A}_{D_v}} \sum_{a_\theta^v=0}^\infty \chi_v(\theta^\vee(\varpi_v)^{a_\theta})q_v^{-(s_\theta^v-\kappa_\theta^v)a_\theta^v l(\theta)}\]

\[=\prod_{v \notin S_D \cup S_F}\prod_{\theta \in \mathcal{A}_v-\mathcal{A}_{D_v}}(1-\chi_v(\theta^\vee(\varpi_v))q_v^{-(s_\theta-\kappa_\theta)l(\theta)})^{-1}.\]

By \cite{JAMS} Proposition 2.9, this infinite product is equal to

\[\prod_{\mathfrak{O} \in \mathcal{A}-\mathcal{A}_D}L^{S_\mathfrak{O}}(s_\mathfrak{O}-\kappa_\mathfrak{O},\xi_\mathfrak{O}(\chi)).\]

Now we turn to $\sum_{v \notin S_D \cup S_F} b_v(\textbf{s})$. Let $\sigma=(\text{Re}(s_\alpha))_\alpha$. Observe that in the definition of $b_v$ we may assume that $\textbf{a} \neq \underline{0}$. Since for each $v \notin S_F \cup S_D$,

$$\{\textbf{a}|\textbf{a} \neq \underline{0}\}=\bigcup_{\theta \in \Delta(G_{F_v},S_v)}\{\textbf{a}:a_\theta^v \neq 0\},$$

we have

\[\sum_{v \notin S_D \cup S_F}|b_v(\textbf{s})| \leq \sum_{v \notin S_D \cup S_F} \sum_\theta \sum_{a_\alpha \neq 0} \delta_v(t_v(\textbf{a}))q_v^{-\langle{\sigma,\textbf{a}\rangle}_v}|(\text{vol}(K_v t_v(\textbf{a})K_v) - \delta_{B_v}(t_v(\textbf{a})))|.\]

\[\ll\sum_{v \notin S_D \cup S_F} q_v^{-1}\sum_{\alpha \notin \mathcal{A}_D}\sum_{a_\alpha \neq 0} q_v^{-\langle{\sigma,\textbf{a}\rangle}_v}\delta_{B_v}(t_v(\textbf{a}))\]

\[=\sum_{v \notin S}q_v^{-1}\sum_{\theta \notin \mathcal{A}_D} (\sum_{a_\theta = 1}^\infty q_v^{-(\sigma_\alpha-\kappa_\theta)a_\theta l(\theta)})\prod_{\beta \neq \theta, \beta \notin \mathcal{A}_D}(\sum_{a_\beta=0}^\infty q_v^{-(\sigma_\beta-\kappa_\beta)a_\beta l(\beta)})\]

\[=\sum_{v \notin S_D \cup S_F}q_v^{-1}\sum_{\theta \notin \mathcal{A}_D}\dfrac{q_v^{-(\sigma_\theta-\kappa_\theta)a_\theta l(\theta)}}{\prod_{\beta \notin \mathcal{A}_D}(1-q_v^{-(\sigma_\beta-\kappa_\beta)l(\beta)})}\]

\[\ll\sum_{\theta \notin \mathcal{A}_D}\sum_{v \notin S_D \cup S_F}q_v^{-3/2}<\infty.\]

We need to show the existence of a $C>0$ such that $|1+a_v| \geq C>0$ for all $v$. For this,

$$|1+a_v| \geq \prod_{\theta \in \mathcal{A}_D}\dfrac{1}{1+q_v^{-\sigma_\beta+\kappa_\beta}} \geq \prod_{\theta \notin \mathcal{A}_D}\dfrac{1}{2} \geq \dfrac{1}{2^r},$$

with $r=|\Delta(G_E,T_E) \backslash \mathcal{A}_D|.$ For $\textbf{s} \in \mathcal{T}^D_{-1/2+\epsilon}$ the estimates are uniform, i.e. the quotient

$$\dfrac{\prod_{v \notin S} I_v(\chi)}{\prod_{v \notin S_D \cup S_F}(1+a_v)}$$

is holomorphic in $\mathcal{T}^D_{-1/2+\epsilon}$. This completes the proof.
\end{proof}

\subsubsection{Places in $S_D \backslash S_F$.}
 
 \begin{proposition}
  For each $v \in S_D \backslash S_D \cap S_F$ there exists a function $f_{\chi_v}$, holomorphic and uniformly bounded in $\mathcal{T}_{-1-\delta}$, for some $\delta>0$, such that

$$\int_{G(F_v)}H_v(\textbf{s},g_v)^{-1}\chi_v(g_v) \ dg_v = \prod_{\mathfrak{O} \in \mathcal{A}}L_{w_\mathfrak{O}}(s_\mathfrak{O}-\kappa_\mathfrak{O},\xi_{\mathfrak{O}}(\chi)_{w_\mathfrak{O}}) \cdot f_{\chi_v}(\textbf{s}),$$

where $L_{w_\mathfrak{O}}$ is the local factor of the Hecke L-function $L(s_\mathfrak{O},\xi_\mathfrak{O}(\chi))$.

 \end{proposition}

\begin{proof} The result follows from the same analysis as for places $v \notin S_F$, with $\delta_v \equiv 1$.
\end{proof}

\subsubsection{Local integrals for places in $S_F$.} 

We turn to places $v \in S_F$. For the analysis of these functions we rely on a general result of \cite{CLT}. If $v \in S_F \notin S_D$, then by Proposition 4.4 in \cite{CLT} the rightmost pole is at

$$\max_{\alpha \notin \mathcal{A}_D}\dfrac{\kappa_\alpha}{\lambda_\alpha}.$$

Such integrals will not contribute to the rightmost pole of the height zeta function. If $v \in S_F \cap S_D$, the analysis in \cite{CLT} shows that the rightmost pole is at

$$\max_{\alpha \in \mathcal{A}}\dfrac{\kappa_\alpha}{\lambda_\alpha}$$

and that the order of the pole is $1+\dim \mathcal{C}^{\text{an}}_{F_v,\lambda}(D)$. They also establish analytic continuation to the left of the pole.

\subsection{Integrals from infinite dimensional representations.}

\begin{lemma}
(1) Let $H$ be a connected reductive group over a number field $F$. Let $v$ be a place of $F$ such that $H(F_v)$ is not compact modulo center. Let $\pi$ be an automorphic representation of $H(\mathbb{A}_F)$. If $\pi_v$ is one-dimensional then $\pi$ is one-dimensional.

(2) Let $F$ be a non-archimedean local field. Let $G$ be a connected reductive group over $F$. If the $F$-simple factors of the derived group of $G$ are $F$-isotropic, then any irreducible smooth representation $V$ of $G(F)$ is either one-dimensional or infinite-dimensional.
\end{lemma}

\begin{proof}
For (1), see Lemma 6.2 in \cite{shin}. For (2), see \cite{conrad}, Proposition 3.9.
\end{proof}

For the integrals

$$\int_{G(\mathbb{A})}\delta_{S,D}(g)H(\textbf{s},g)^{-1}E(g,\phi) \ dg,$$

we will follow and briefly sketch the argument presented in section 4.5 of \cite{loughran}. At any finite place $v$ we let $C_c^\infty(G(F_v))$ denote, as usual, the space of functions on $G(F_v)$ that are locally constant and of compact support. For archimedean $v$ we require such functions to be smooth and of compact support. The set $C_c^\infty(G(F_v))$ forms a convolution algebra $\mathcal{H}(G(F_v))$ with respect to the measure $dg_v$. For each place $v$ we define idempotents $\xi_v$ as in section 4.5 of \cite{loughran}. The global Hecke algebra $\mathcal{H}(G(\mathbb{A}))$ is the space of finite linear combinations of functions $\otimes_v \varphi_v$, where $\varphi_v \in \mathcal{H}(G(F_v))$ and $\varphi_v$ is $\xi_v$ for almost all $v$. 

Since $\delta \cdot H$ is invariant on the left and right under the compact open subgroup $K_v$ for each non-archimedean place $v$, there is an associated idempotent $\xi_0=\otimes'_{v \text{non-arch.}}\xi_v$ in the Hecke algebra $\otimes'_{v \text{non-arch.}}\mathcal{H}(G(k_v))$ such that $\xi_0 * (\delta \cdot H)=(\delta \cdot H) * \xi_0$. By a theorem of Harish-Chandra, $\sum_{W \in \hat{K}_\infty} \xi_W * H_\infty$ converges in the topology of $C^\infty(G_\infty)$ to $H_\infty$. Therefore, we may rewrite the integral above as

$$\int_{G(\mathbb{A})}(\sum_{W \in \hat{K}_\infty}\xi_W \otimes \xi_0)H(\textbf{s},g)^{-1}E(g,\phi) \ dg =\sum_{W \in \hat{K}_\infty}\int_{G(\mathbb{A})}(\xi_W \otimes \xi_0) * H(\textbf{s},g))^{-1}E(g,\phi) \ dg$$

$$=\sum_{W \in \hat{K}_\infty}\int_{G(\mathbb{A})}H(g,\textbf{s})^{-1}(\xi_W \otimes \xi_0) * E(g,\phi) \ dg.$$

We define

$$M_\xi(g,\phi)=(\xi * E)(g,\phi).$$

For groups of rank at least two we will use uniform bounds on matrix coefficients obtained by H. Oh. The same result is expected to hold in general by adapting the proof of Theorem 4.5 in \cite{JAMS}.

We shall use the following result - see \cite{loughran}, Lemma 4.7. Suppose we are given a strongly orthogonal system $\mathcal{S}_v$ in $G(k_v)$ for each place $v$. We continue to use the notation in section 4.5 of \cite{loughran}. Let $\xi$ be a non-trivial idempotent in the global Hecke algebra. 

\begin{lemma}
There is a constant $C_\xi$, depending only on the idempotent $\xi$, such that
\end{lemma}

$$|M_\xi(g,\phi)| \leq C_\xi \sqrt{\text{dim}H_P(\pi)_{\chi,K_0,W}} \cdot \text{max}_{\phi \in B_P(\pi)_\chi \cap H_P(\pi)_{\chi,K_0,W}}\{|E(e,\phi)|\} \cdot \prod_v \xi_{\mathcal{S}_v}(g_v).$$

\begin{proposition}
Let $r$ denote the rank of $G$. Given $\epsilon>0$, there is a constant $C_{\xi,\epsilon}$, depending only on $\epsilon$ and the idempotent $\xi$, such that

\end{proposition}

$$|M_\xi(g,\phi)| \leq C_{\xi,\epsilon} \sqrt{\text{dim}H_P(\pi)_{\chi,K_0,W}} \cdot \text{max}_{\phi \in B_P(\pi)_\chi \cap H_P(\pi)_{\chi,K_0,W}}\{|E(e,\phi)|\}$$

$$\cdot \prod_v \prod_{\theta_v \in \Delta(G_{F_v},S_v)}|\theta_v(t_v)|_v^{-l_v(\theta)/(2r)+\epsilon}.$$

\begin{proof}
First, we fix a place $v$ of $F$. Each root $\theta \in \Delta(G_{F_v},S_v)$ forms a strongly orthogonal system. By Theorem 5.9(3) of \cite{oh}, for every $\epsilon>0$, there is a constant $C_\epsilon$ such that for every $a \in S_v^+$, 

$$\xi_\mathcal{S}(a) \leq C_\epsilon |\theta(a)|_v^{-1/2+r\epsilon}$$

Multiplying these inequalities over all simple roots of $G$ over $E$  and taking the $r$th root gives the result.
\end{proof}

Let $\varphi$ be an automorphic form of $G(\mathbb{A})$ in the space of an automorphic representation $\pi$ which is right invariant under the maximal compact subgroup $K$. We must bound the infinite product $\prod_v I_v(\textbf{s},\varphi)$, where

$$I_v(\textbf{s},\varphi)=\int_{G(F_v)}\delta_{S,D}(g_v)H_v(\textbf{s},g_v)^{-1}\prod_{\theta_v \in \Delta(G_{F_v},S_v)}|\theta_v(t_v)|_v^{-l_v(\theta_v)/(2r)+\epsilon} \ dg_v$$

\subsubsection{Local integrals.}

\begin{proposition}
For all $v \notin S_\infty$, the integral $I_v(\textbf{s},\varphi)$ is holomorphic for $\textbf{s} \in \mathcal{T}_{-1-1/(2r)}.$ Moreover, for all $\epsilon>0$ there is a constant $C_v(\epsilon)$ such that $|I_v(\textbf{s},\varphi)| \leq C_v(\epsilon)$ for all $\textbf{s} \in \mathcal{T}_{-1-1/(2r)+\epsilon}$.

(2) For $v \in S_\infty$ and $\partial$ in the universal enveloping algebra the integral

$$I_{v,\partial}(\textbf{s},\varphi)=\int_{G(F_v)}\partial(H_v(\textbf{s},g_v)^{-1})\prod_{\theta_v \in \Delta(G_{F_v},S_v)}|\theta_v(t_v)|_v^{-l_v(\theta_v)/(2r)+\epsilon}\ dg_v$$

is holomorphic for $\textbf{s} \in \mathcal{T}_{-1-1/(2r)}.$ Moreover, for all $\epsilon>0$ there is a constant $C_v(\partial,\epsilon)$ such that $|I_{v,\partial}(\textbf{s},\varphi_{\pi_v})| \leq C_v(\partial,\epsilon)$ for all $\textbf{s} \in \mathcal{T}_{-1-1/(2r)+\epsilon}$.
\end{proposition}

\begin{proof}
We will only prove the first part; the second part is similar. Locally, every two local integral structures give rise to essentially equivalent height functions; so, we replace the local integral structure so that the resulting height function is invariant under $K_v$, a good maximal compact subgroup. Let $\underline{\sigma}$ be the vector consisting of the real parts of the components of $\textbf{s}.$ The local height integral is majorized by

$$\prod_{\theta_v \in \Delta(G_{F_v},S_v)}\sum_{l=0}^\infty \delta_{B_v}(\theta_v^\vee(\varpi_v^l))H(\underline{\sigma},\theta_v^\vee(\varpi_v^l))^{-1}q_v^{-(1/(2r)-\epsilon)l}$$

$$\prod_{\theta_v \in \Delta(G_{F_v},S_v)}\sum_{l=0}^\infty q_v^{-(\sigma_{\theta_v}-\kappa_{\theta_v}+1/(2r)-\epsilon)ll_v(\theta_v)}.$$

\end{proof}



\subsubsection{Infinite product.}

\begin{proposition}

The infinite product

\begin{equation}
I_{S,D}(\textbf{s},\varphi)=\prod_{v \notin S_F \cup S_D} I_v(\textbf{s},\varphi)
\end{equation}
    
is holomorphic for $\textbf{s} \in \mathcal{T}^D_{-1/(2r)}$. Moreover, for all $\epsilon>0$ and all compact subsets $\underline{K} \subset \mathcal{T}^D_{-1/(2r)+\epsilon}$ there exists a constant $C(\epsilon,\underline{K})$, independent of $\pi$, such that for all $\textbf{s} \in \underline{K}$,

\begin{equation}
|I_{S,D}(\textbf{s},\varphi)| \leq C(\epsilon,\underline{K}).
\end{equation}
\end{proposition}

\begin{proof}

For each vector $\textbf{a}=(a_\alpha)_\alpha \in \mathcal{T}(\mathbb{N})$, we set

\[t_v(\textbf{a})=\prod_{\theta_v \in \Delta(G_{F_v},S_v)}\theta_v^\vee(\varpi_v)^{a_{\theta_v}}.\]

Let $\epsilon>0$. $I_{S,D}(\textbf{s},\varphi)$ is bounded by

\[\prod_{v \notin S_D \cup S_F} \sum_\textbf{a} \delta_v(t_v(\textbf{a}))\delta_{B_v}(t_v(\textbf{a}))\prod_{\theta_v \in \Delta(G_{F_v},S_v)} (|\theta_v(t_v(\textbf{a}))|_v^{-l_v(\theta)/2r+\epsilon-l_v(\theta)s_{\theta_v}}).\]


Therefore, to establish the convergence of the Euler product over places $v \notin S_D \cup S_F$ it suffices to bound

\[\sum_{v \notin S_D \cup S_F}\sum_{(a_\alpha) \in \mathbb{N}^{\mathcal{A} \backslash \mathcal{A}_D}} q_v^{-\sum_{\theta \notin \mathcal{A}_{D_v}} a_{\theta_v}[(s_{\theta_v}-\kappa_{\theta_v}+1/(2r))l_v(\theta_v)-\epsilon]}.\]

\end{proof}

\begin{corollary}
$I_{S,D}(\textbf{s},\varphi)$ has an analytic continuation to a function which is holomorphic on $\mathcal{T}^D_{-1/2} \cap \mathcal{T}_{-1-1/(2r)}$. Suppose $\varphi$ is an eigenfunction for $\Delta$. Define $\Lambda(\varphi)$ by $\Delta \varphi = \Lambda(\varphi) \varphi$. Then for each integer $k>0$, all $\epsilon>0$, and every compact subset $\underline{K} \subset \mathcal{T}^D_{-1/(2r)+\epsilon} \cap \mathcal{T}_{-1-1/(2r)+\epsilon},$ there exists a constant $C=C(\epsilon,\underline{K},k)$, independent of $\phi$, such that for all $\textbf{s} \in \underline{K}$,

\begin{equation}
|I_{S,D}(\textbf{s},\varphi)| \leq C \Lambda(\varphi)^{-k}|\varphi(e)|.
\end{equation}

\end{corollary}







The following proposition shows that infinite-dimensional representations will never contribute to the right-most pole of the height zeta function.

\begin{proposition} The function $S^\flat$ admits an analytic continuation to a function which is holomorphic on $\mathcal{T}^D_{-1/2r} \cap \mathcal{T}_{-1-1/2r}$, where $r$ is the rank of $G$.
\end{proposition} 

\begin{proof}
We need to show the convergence of

$$\sum_{\chi \in \mathfrak{X}}^\flat \sum_P n(A_P)^{-1}\sum_{W \in \hat{K}_\infty}\int_{\Pi(M_P)}(\sum_{\phi \in B_P(\pi)_\chi \cap H_P(\pi)_{\chi,K_0,W}}\Lambda(\phi)^{-r}|E(e,\phi)|\sqrt{\text{dim}H_P(\pi)_{\chi,K_0,W}}$$

$$\times \max_{\phi \in B_P(\pi)_\chi \cap H_P(\pi)_{\chi,K_0,W}}\{|E(e,\phi)|\}) \ d\pi$$

for $r$ large. The proof of this is in the proof of Theorem 4.10 in \cite{loughran}.
\end{proof}

\subsection{Characters that contribute to the leading pole}

We now determine which characters contribute to the main pole. Let $\mathcal{X}_{S,D,\lambda}(G) \subset \mathcal{X}(G)$ be the collection of all characters $\chi=\otimes'_v \chi_v$ such that

- for all $\alpha \in \mathcal{A}(\lambda) - \mathcal{A}_D(\lambda)$ and all $v \notin S$, we have $\xi_\alpha(\chi)_w \equiv 1$; 

- for all $\alpha \in \mathcal{A}_D(\lambda)$ and $v \in S_D \backslash S_D \cap S_F$, we have $\xi_\alpha(\chi)_w \equiv 1$.

- for all $v \in S_D,$ we have $\chi_v|_{D_v(F_v)} \not \equiv 0$.

\subsection{The leading pole} We have shown that

$$Z_{S,D}(s\lambda)=\sum_{\chi \in \mathcal{X}(G)}\int_{G(\mathbb{A})} \delta_{S,D}(g)H(s\lambda,g)^{-1}\chi(g) \ dg + f(s)$$

with $f$ holomorphic for Re($s)>a(\lambda)-\delta$, for some $\delta>0$. For $\chi \in \mathcal{X}(G)$ the integral

\[\int_{G(\mathbb{A})} \delta_{S,D}(g)H_S(s\lambda,g)^{-1}\chi(g) \ dg\]
admits a regularization of the shape

$$\prod_{\alpha \in \mathcal{A}(\lambda)-\mathcal{A}_D(\lambda)}L_S(s\lambda_\alpha-\kappa_\alpha,\xi_\alpha(\chi')) \cdot h_\chi(s) \cdot \prod_{v \in S_D \backslash S_D \cap S_F} \prod_{\alpha \in \mathcal{A}_D(\lambda)}L_v(s\lambda_\alpha-\kappa_\alpha,\xi_{\alpha,v}(\chi'_v)) \cdot h_{\chi,v}(s),$$

with $h_\chi$ and $h_{\chi,v}$ holomorphic for Re($s)>a(\lambda)-\delta,$ for some $\delta>0$. It follows that only $\chi \in \mathcal{X}_{S,D,\lambda}(G)$ contribute to the leading term at $s=a(\lambda)$. We can rewrite this contribution as 

$$|\mathcal{X}_{S,D,\lambda}(G)| \int_{G(\mathbb{A})^{\text{Ker}_\lambda}}\delta_{S,D}(g)H(s\lambda,g)^{-1} \ dg,$$

where 

    $$G(\mathbb{A})^{\text{Ker}_\lambda}=\bigcap_{\chi \in \mathcal{X}_{S,D,\lambda}(G)}\text{Ker}(\chi)$$

is the intersection of the kernels of automorphic characters. Theorem 1 follows exactly as in the proof of Theorem 6.4 in \cite{TTB}.









\subsubsection{Restriction to the anticanonical line.}

We now specialize to the case of the log-anticanonical line bundle. Recall that $D=\bigcup_{\alpha \in \mathcal{A}_D}D_\alpha$. Also, $-K_X=\sum_\alpha (\kappa_\alpha+1)D_\alpha$. Then

$$
\max_{\alpha \in \mathcal{A}_D}\dfrac{\kappa_\alpha}{\lambda_\alpha}=\max_{\alpha \notin \mathcal{A}_D}\dfrac{\kappa_\alpha+1}{\lambda_\alpha}=1.$$

Thus the right-most pole for $-K_X-D$ is at $s=1$. In this case, each place $v \in S$ contributes, to the pole at $s=1$, a pole of order $1+ \dim \mathcal{C}^{an}_{F_v}(D)$. The order of the pole at $s=1$ is

$$b_=\mathrm{rank}(\mathrm{Pic}(X \backslash D))+\sum_{v \in S_D}(1+\dim \mathcal{C}_{F_v}^{an}(D)).$$

We describe the constant $c$ appearing in the statement of Theorem 1 in terms of Tamagawa-type constants. For places $v \in S$, assuming that $D(F_v) \neq \emptyset,$ there is a distinguished $v$-adic measure $\tau_v^{\max}$ on $D(F_v)$; we refer to \cite{CLT} for the definition. The corresponding volumes are given by

$$\tau_v^{\max}(D(F_v))=\prod_{\alpha \in \mathcal{A}_D}\dfrac{1}{\kappa_\alpha}\lim_{s \rightarrow 1}(s-1)^d \int_{G(F_v)}H_v(s\lambda,g_v)^{-1} \ dg_v.$$

There is an adelic measure on the integral adeles on $U=X \backslash D$. This measure has the form

$$\tau^S_{(X,D)}(U(\mathbb{A}_S)^{\text{int}})=\prod_{\alpha \notin \mathcal{A}_D} \dfrac{1}{\kappa_\alpha+1} \lim_{s \rightarrow 1}(s-1)^{r-d}\int_{G(\mathbb{A}_S)} \delta_{S}(g)H_S(s\lambda,g_v)^{-1} \ dg.$$

It follows that

$$c=\dfrac{1}{(b-1)!} \cdot |\chi_{S,D,\lambda}(G)| \cdot  \tau_{(X,D)}^S(U(\mathbb{A}_S)^{\text{int}}) \cdot \prod_{v \in S}\tau_v^{\max}(D(F_v)).$$

\section{Acknowledgments.} It is a pleasure to thank my thesis adviser, R. Takloo-Bighash, who first suggested the topic of research to me and provided helpful guidance and encouragement over the years. The author would also like to thank Y. Tschinkel and P. Sarnak for their interest and stimulating comments. Thanks are also due to D. Prasad and S.W. Shin for answering a question the author had about automorphic representations.







\bibliographystyle{amsplain}

\end{document}